\documentclass[11pt]{amsart}
\usepackage{geometry}                
\usepackage{graphicx}
\usepackage{amssymb}
\usepackage{epstopdf}
\usepackage{float}
\usepackage{amsmath,amsfonts,amsthm,url,array,etoolbox}
\usepackage{enumerate}
\usepackage{tikz}
\usetikzlibrary{calc,decorations.markings}
\theoremstyle{plain}
\newtheorem{thm}{Theorem}[section]
\newtheorem{lem}[thm]{Lemma}
\newtheorem{prop}[thm]{Proposition}

\theoremstyle{definition}

\theoremstyle{remark}
\newtheorem{rem}{Remark}

\begin{document}
\title[]{The relative cohomology of abelian covers of the flat pillowcase}
\author[]{Chenxi Wu}
\maketitle
\begin{abstract}
We calculate the action of the group of affine diffeomorphisms on the relative cohomology of pair $(M,\Sigma)$, where $M$ is a square-tiled surface that is a normal abelian cover of the flat pillowcase, and as an application, answer a question raised by Smillie and Weiss.
\end{abstract}
\section{Introduction}

The flat pillowcase has a half translation structure, which induces a half-translation structure on its branched covers. We give a comprehensive treatment of the relative cohomology of branched abelian normal covers of the flat pillowcase, the action of the affine diffeomorphism group of these branched cover on their relative cohomology, as well as invariant subspaces under this group action. These subspaces are orthogonal under an invariant Hermitian norm. The corresponding question for absolute cohomology is a classical one related to the monodromy of the hypergeometric functions which dates back to Euler and is outlined in \cite{dm86} (see also \cite{wright2012}). Most previous work on this topic focuses on the case of absolute cohomology, and uses holomorphic methods. Due to the recent interest in translation surfaces and $SL(2,\mathbb{R})$-orbit closures in them, there is also interest in the relative cohomology $H^1(M,\Sigma;\mathbb{C})$ because the relative cohomology is the tangent space of strata, and because these branched covers are in $SL(2,\mathbb{R})$ closed orbits. For example, Matheus and Yoccoz \cite{my09} calculate the action of the full affine group on the relative cohomology of two well-known translation surfaces, the Wollmilchsau and the Ornithorynque, both of which are examples of abelian branched covers of the flat pillowcase. In this paper, we will give a complete description of the action of the full affine group on relative cohomology of all abelian branched covers of the pillowcase.\\

We will begin by showing that there is a direct sum decomposition of the relative cohomology, and calculate the dimension of the summands.

\begin{thm}
Let $M$ be a branched cover of the pillowcase with deck group $G$. Let $\Sigma\subset M$ be the preimage of the four cone points of the pillowcase. Let $\Delta$ be the set of irreducible representations of a finite abelian group $G$. These are all one dimensional, hence are homomorphisms from $G$ to $\mathbb{C}^*$. Deck transformations give an action of $G$ on $H^1(M,\Sigma;\mathbb{C})$. Let $H^1(\rho)$ be the sum of $G$-submodules of $H^1(M,\Sigma;\mathbb{C})$ isomorphic to $\rho$. 
\begin{enumerate}
\item  
$$H^1(M,\Sigma;\mathbb{C})=\bigoplus_{\rho\in\Delta} H^1(\rho)$$
this decomposition is preserved by the action of the affine group $\mathbf{Aff}$, while the factors may be permuted. 
\item The dimension of $H^1(\rho)$ is 3 if $\rho$ is the trivial representation and 2 otherwise.
\end{enumerate}
\end{thm}

The space $H^1(\rho)$ can also be described as cohomology of the pillowcase with twisted coefficients as in \cite{dm86}, \cite{thruston98}. Let $H^1_{abs}(\rho)$ be the sum of $G$-submodules of $H^1(M;\mathbb{C})$ isomorphic to $\rho$, then there is also a splitting $H^1(M)=\bigoplus_{\rho}H^1_{abs}(\rho)$ which was described in \cite{wright2012}.\\

There is a natural projection $r: H^1(M,\Sigma;\mathbb{C})\rightarrow H^1(M;\mathbb{C})$ induced by the inclusion $(M,\emptyset)\rightarrow(M,\Sigma)$, which is also equivariant under the deck group $G$. The kernel of $r$ is called the $rel$-space. It is interesting to know when this projection splits, in other words, when $rel$ has an $\mathbf{Aff}$-invariant complement.\\ 

Because $r$ is $G$-equivariant and surjective, it sends $H^1(\rho)$ surjectively to $H^1_{abs}(\rho)$. For abelian branched covers of the pillowcase, we can describe $r$ by describing its restriction on each $H^1(\rho)$. The result can be summarized as follows:
\begin{thm}Let $M$, $\Sigma$, $r$ be as above, then:
\begin{enumerate}
\item If two or four of the four $\rho(g_j)$ are equal to $1$ then $r|_{H^1(\rho)}=0$, hence $H^1_{abs}(\rho)=0$.
\item If only one of the four $\rho(g_j)$ is $1$ then $\ker(r|_{H^1(\rho)})\rightarrow H^1(\rho)\rightarrow r(H^1(\rho))$ does not split as a $\Gamma$-module. In this case $H^1_{abs}(\rho)$ has dimension 1.
\item In all other cases, $r|_{H^1(\rho)}$ is bijective hence splits. In this case $H^1_{abs}(\rho)$ has dimension $2$.
\end{enumerate}
In summary, $r: H^1(M,\Sigma;\mathbb{C})\rightarrow H^1(M;\mathbb{C})$ splits if and only if case (2) does not appears, i.e. if and only if for any $i\in\{1,2,3,4\}$, either the subgroup generated by $g_i$, $(g_i)$, contains $g_j$ for some $j\neq i$, or $G/(g_i)=(\mathbb{Z}/2)^2$. 
\end{thm}

For example, when $M$ is the Wollmilchsau, $r$ splits, because $(g_1)=(g_2)=(g_3)=(g_4)=G$.\\

The dimension of $H^1_{abs}(\rho)$ was calculated in \cite{wright2012}.\\

There is a natural invariant Hermitian form on $H^1(M,\Sigma;\mathbb{C})$ (see \cite{dm86}, \cite{thruston98} or Section 4). In case (1) this form is trivial, in case (2) this form induces an Euclidean structure on its complex projectivization $\mathbb{P}(H^1(\rho))$, and in case (3) it may be positive definite, negative definite or indefinite, depending on the arguments of $\rho(g_i)$, hence induces either a spherical or hyperbolic structure on $\mathbb{P}(H^1(\rho))$. The relation between the signature of this Hermitian form on $\rho(g_i)$ and the arguments of $\rho(g_i)$ is a consequence of Riemann-Roch. It was given in \cite{dm} and is included here as Theorem 4.1.\\

In section 5 we will describe a subgroup of finite index $\Gamma_1$ of the affine diffeomorphism group. The subspaces $H^1(\rho)$ are invariant under $\Gamma_1$ and the action of $\Gamma_1$ can be described as follows:\\ 

\begin{thm}If $\rho$ satisfy the conditions in case (2) and (3) of Theorem 1.2, the action of $\Gamma_1\subset\mathbf{Aff}$ on $\mathbb{P}H^1(\rho)=\mathbb{CP}^1$ is through an index-2 subgroup of a (Euclidean, hyperbolic or spherical) triangle group.
\end{thm}

By considering the angles of the triangles, we can easily determine when $\Gamma_1$, hence the $\mathbf{Aff}$, acts discretely on $H^1(\rho)$.\\
 
As an application of these results, we will construct examples, which answer a question of Smillie and Weiss. In \cite{sw}, they use these examples to show that horocycle orbit closures in strata may be non-affine. For their examples, Smillie and Weiss require that there is an invariant subspace of $H^1(M,\Sigma;\mathbb{C})$ defined over $\mathbb{R}$. Define complex conjugation in $H^1(M,\Sigma;\mathbb{C})$ by the complex conjugate in coefficient field $\mathbb{C}$. For any subspace $N\subset H^1(M,\Sigma;\mathbb{C})$, let $\overline{N}$ be the complex conjugate of $N$. A complex subspace $N$ is defined over $\mathbb{R}$ i.e. $N=Re(N)\otimes\mathbb{C}$, if and only if $\overline{N}=N$.

\begin{prop} There is a square-tiled surface $M$ on which all points in $\Sigma$ are singular, constructed as a normal abelian branched cover of the flat pillowcase such that: 
\begin{enumerate}
\item there is a direct sum decomposition $H^1(M,\Sigma;\mathbb{C})=N\oplus\overline{N}\oplus H$ preserved by the action of the group of orientation preserving affine diffeomorphisms. Furthermore, $H$ is defined over $\mathbb{R}$.
\item there is a positive or negative definite Hermitian norm on $N$ invariant under the affine diffeomorphism group action
\item the affine diffeomorphism group action on $N$ does not factor through a discrete group.
\end{enumerate}
\end{prop}

\begin{rem} The decomposition $H^1(M,\Sigma;\mathbb{C})=N\oplus\overline{N}\oplus H$ implies that $H^1(M,\Sigma;\mathbb{R})=Re(N)\oplus Re(H)$. Furthermore, $N$, $\overline{N}$ and $H$ are orthogonal to each other with respect to the invariant Hermitian form, and the tangent space of the $GL(2,\mathbb{R})$ orbit of $M$ lies in $H$\end{rem}

Hubert-Schmith{\"u}sen \cite{hs} gave a proof of the non-discreteness of the action of the affine group in some cases using Lyapunov exponents and Galois conjugate. Forni-Matheus-Zorich \cite{fmz}, Bouw-M\"oller \cite{bm}, Deligne-Mostow \cite{dm86}, Thurston \cite{thruston98}, Alex Wright \cite{wright2012}, McMullen \cite{mcmullen13} and Eskin-Kontsevich-Zorich \cite{ekz} calculated the action on cohomology and provided discreteness criteria under different contexts.\\

The author thanks his thesis advisor John Smillie for suggesting the problem and many helpful conversations. The author also thanks Alex Wright, Barak Weiss and Gabriela Schmith{\"u}sen for many helpful comments.\\

As pointed out by Alex Wright, the existence of examples answering the question of Smillie and Weiss follows from the following ingredients: firstly, Theorem 1.1, secondly, a signature calculation of the Hodge form on each component, and thirdly, a discreteness criteria. The discreteness criteria and signature calculation in \cite{dm86}, together with Theorem 1.1, are already enough for the construction of many such examples. We will give a description of the affine diffeomorphism group of these surfaces in Section 2, and define the group $\Gamma$. In Section 3, we will prove Theorem 1.1. In Section 4, we define and calculated the signature of an invariant Hermitian form. In Section 5 we will define $\Gamma_1$ and prove Theorem 1.2 and 1.3. In Section 6 we will discuss the discreteness criteria and construct examples that establish Proposition 1.4.\\

We will now set up some notation to describe normal branched covers of the pillowcase. Let $P$ be the unit flat pillowcase with four cone points $z_1$, $z_2$, $z_3$ and $z_4$ of cone angle $\pi$. We build $P$ by identifying edges with the same label as in Figure 1:

\begin{figure}[H]
\begin{tikzpicture}[scale=2,  ray/.style={decoration={markings,mark=at position .5 with {
      \arrow[>=latex]{>}}},postaction=decorate}
]
\draw[ray] (0,1)--(0,2) node[pos=.5,left]{$e^2$} node[pos=1,left]{$z_3$};
\draw[ray] (0,2)--(1,2) node[pos=.5,above]{$e^3$} node[pos=1,right]{$z_4$};
\draw[ray] (1,2)--(1,1) node[pos=.5,right]{$e^4$};
\draw[ray] (1,1)--(0,1) node[pos=0,right]{$z_1$} node[pos=.5,above] {$e^1$} node[pos=1,left]{$z_2$};
\draw[ray] (0,1)--(0,0) node[pos=.5,left]{$e^2$};
\draw[ray] (0,0)--(1,0) node[pos=0,left]{$z_3$} node[pos=.5,below] {$e^3$} node[pos=1,right]{$z_4$};
\draw[ray] (1,0)--(1,1) node[pos=.5,right]{$e^4$};
\node at (.5,1.5){$B^2$};
\node at (.5,.5){$B^1$};
\end{tikzpicture}
\caption{\label{p}}
\end{figure} 

Let $G$ be a finite group and $\mathbf{g}=(g_1,\dots,g_4)\in G^4$ a 4-tuple of elements in $G$ such that $g_1g_2g_3g_4=1$. Let $M=M(G,\mathbf{g})$ be the connected normal branched cover of $P$ branching at $z_1,\dots z_4$, with deck transformation group $G$ acting on the left. Let $l_j$ be a simple loop around $z_j$ that travels in counter-clockwise direction on $P$ based in $B^1$. It lifts to a path from the preimage of $B^1$ in the $g$-th sheet of the cover to the preimage of $B^1$ in the $gg_j$-th sheet. In other words, $\mathbf{g}$ gives a group homomorphism from 
$$\pi_1(P-\{z_1,z_2,z_3,z_4\})=\langle l_1,l_2,l_3,l_4|l_1l_2l_3l_4=1\rangle$$
to $G$. Here the homomorphism defined by $\mathbf{g}$ sends the element in $\pi_1(P-\{z_1,z_2,z_3,z_4\})$ represented by $l_j$ to $g_j\in G$. The connectedness of $M$ is equivalent to the condition that $\{g_1,\dots,g_4\}$ generate $G$. Let $\Sigma$ denote the set of preimages of all points $z_j$, $j=1,\dots,4$. The surface $M$ has a half translation structure induced by the half translation structure on $P$.\\

When the orders of $g_j$ are all even, all the holonomies are translations and $M$ is a translation surface. When the order of $g_j$ is 2, the corresponding vertex has cone angle $2\pi$. When none of the orders of $g_j$ is 2, $\Sigma$ consists of actual cone points of $M$, in which case $\mathbf{Aff}$ is the affine diffeomorphism group. \\

The decomposition of $P$ into two squares in figure~\ref{p} induces a cell decomposition on $M(G,\mathbf{g})$, which can be described as $|G|$-copies of pairs of squares labeled by elements in $G$ as $B^1_g$, $B^2_g$, that are glued together by identifying edges $e^j_g$ and $e^{j'}_{g'}$ when $j=j'$ and $g=g'$, so that the directions indicated by the arrows match:

\begin{figure}[H]
\begin{tikzpicture}[scale=2,  ray/.style={decoration={markings,mark=at position .5 with {
      \arrow[>=latex]{>}}},postaction=decorate}
]

\draw[ray] (0,1)--(0,2) node[pos=.5,left]{$e^2_{gg_2}$} node[pos=1,left]{$z_3$};
\draw[ray] (0,2)--(1,2) node[pos=.5,above]{$e^3_{gg_2g_3}$} node[pos=1,right]{$z_4$};
\draw[ray] (1,2)--(1,1) node[pos=.5,right]{$e^4_{gg_2g_3g_4}$};
\draw[ray] (1,1)--(0,1) node[pos=0,right]{$z_1$} node[pos=.5,above] {$e^1_g$} node[pos=1,left]{$z_2$};
\draw[ray] (0,1)--(0,0) node[pos=.5,left]{$e^2_g$};
\draw[ray] (0,0)--(1,0) node[pos=0,left]{$z_3$} node[pos=.5,below] {$e^3_g$} node[pos=1,right]{$z_4$};
\draw[ray] (1,0)--(1,1) node[pos=.5,right]{$e^4_g$};
\node at (.5,1.5){$B^2_g$};
\node at (.5,.5){$B^1_g$};

\end{tikzpicture}
\caption{\label{p2}}
\end{figure} 

For example,  in our notation the Wollmilchsau \cite{f06,hs08}
 is $M(\mathbb{Z}/4, (1,1,1,1))$, can be presented as the union of the following squares with indicated glueings :

\begin{figure}[H]
\begin{tikzpicture}[scale=1.5,  ray/.style={decoration={markings,mark=at position .5 with {
      \arrow[>=latex]{>}}},postaction=decorate}
]

\draw[ray] (0,1)--(0,2) node[pos=.5,left]{$e^2_{1}$} node[pos=1,left]{$z_3$};
\draw[ray] (0,2)--(1,2) node[pos=.5,above]{$e^3_{2}$} node[pos=1,right]{$z_4$};
\draw[ray] (1,2)--(1,1) node[pos=.5,right]{$e^4_{3}$};
\draw[ray] (1,1)--(0,1) node[pos=0,right]{$z_1$} node[pos=1,left]{$z_2$};
\draw[ray] (0,1)--(0,0) node[pos=.5,left]{$e^2_0$};
\draw[ray] (0,0)--(1,0) node[pos=0,left]{$z_3$} node[pos=.5,below] {$e^3_0$} node[pos=1,right]{$z_4$};
\draw[ray] (1,0)--(1,1) node[pos=.5,right]{$e^4_0$};
\node at (.5,1.5){$B^2_0$};
\node at (.5,.5){$B^1_0$};

\draw[ray] (2,1)--(2,2) node[pos=.5,left]{$e^2_{2}$} node[pos=1,left]{$z_3$};
\draw[ray] (2,2)--(3,2) node[pos=.5,above]{$e^3_{3}$} node[pos=1,right]{$z_4$};
\draw[ray] (3,2)--(3,1) node[pos=.5,right]{$e^4_{0}$};
\draw[ray] (3,1)--(2,1) node[pos=0,right]{$z_1$} node[pos=1,left]{$z_2$};
\draw[ray] (2,1)--(2,0) node[pos=.5,left]{$e^2_1$};
\draw[ray] (2,0)--(3,0) node[pos=0,left]{$z_3$} node[pos=.5,below] {$e^3_1$} node[pos=1,right]{$z_4$};
\draw[ray] (3,0)--(3,1) node[pos=.5,right]{$e^4_1$};
\node at (2.5,1.5){$B^2_1$};
\node at (2.5,.5){$B^1_1$};

\draw[ray] (4,1)--(4,2) node[pos=.5,left]{$e^2_{3}$} node[pos=1,left]{$z_3$};
\draw[ray] (4,2)--(5,2) node[pos=.5,above]{$e^3_{0}$} node[pos=1,right]{$z_4$};
\draw[ray] (5,2)--(5,1) node[pos=.5,right]{$e^4_{1}$};
\draw[ray] (5,1)--(4,1) node[pos=0,right]{$z_1$} node[pos=1,left]{$z_2$};
\draw[ray] (4,1)--(4,0) node[pos=.5,left]{$e^2_2$};
\draw[ray] (4,0)--(5,0) node[pos=0,left]{$z_3$} node[pos=.5,below] {$e^3_2$} node[pos=1,right]{$z_4$};
\draw[ray] (5,0)--(5,1) node[pos=.5,right]{$e^4_2$};
\node at (4.5,1.5){$B^2_2$};
\node at (4.5,.5){$B^1_2$};

\draw[ray] (6,1)--(6,2) node[pos=.5,left]{$e^2_{0}$} node[pos=1,left]{$z_3$};
\draw[ray] (6,2)--(7,2) node[pos=.5,above]{$e^3_{1}$} node[pos=1,right]{$z_4$};
\draw[ray] (7,2)--(7,1) node[pos=.5,right]{$e^4_{2}$};
\draw[ray] (7,1)--(6,1) node[pos=0,right]{$z_1$}  node[pos=1,left]{$z_2$};
\draw[ray] (6,1)--(6,0) node[pos=.5,left]{$e^2_3$};
\draw[ray] (6,0)--(7,0) node[pos=0,left]{$z_3$} node[pos=.5,below] {$e^3_3$} node[pos=1,right]{$z_4$};
\draw[ray] (7,0)--(7,1) node[pos=.5,right]{$e^4_3$};
\node at (6.5,1.5){$B^2_3$};
\node at (6.5,.5){$B^1_3$};
\end{tikzpicture}
\caption{\label{p3}}
\end{figure} 

As another example, let $G=\mathbb{Z}/3$ and $\mathbf{g}=(0,1,1,1)$. In this case $M=M(G,\mathbf{g})$ is a half translation surface and the gluing is as follows:
\begin{figure}[H]
\begin{tikzpicture}[scale=1.5,  ray/.style={decoration={markings,mark=at position .5 with {
      \arrow[>=latex]{>}}},postaction=decorate}
]

\draw[ray] (0,1)--(0,2) node[pos=.5,left]{$e^2_1$} node[pos=1,left]{$z_3$};
\draw[ray] (0,2)--(1,2) node[pos=.5,above]{$e^3_2$} node[pos=1,right]{$z_4$};
\draw[ray] (1,2)--(1,1) node[pos=.5,right]{$e^4_0$};
\draw[ray] (1,1)--(0,1) node[pos=0,right]{$z_1$} node[pos=1,left]{$z_2$};
\draw[ray] (0,1)--(0,0) node[pos=.5,left]{$e^2_0$};
\draw[ray] (0,0)--(1,0) node[pos=0,left]{$z_3$} node[pos=.5,below] {$e^3_0$} node[pos=1,right]{$z_4$};
\draw[ray] (1,0)--(1,1) node[pos=.5,right]{$e^4_0$};
\node at (.5,1.5){$B^2_0$};
\node at (.5,.5){$B^1_0$};

\draw[ray] (2,1)--(2,2) node[pos=.5,left]{$e^2_2$} node[pos=1,left]{$z_3$};
\draw[ray] (2,2)--(3,2) node[pos=.5,above]{$e^3_0$} node[pos=1,right]{$z_4$};
\draw[ray] (3,2)--(3,1) node[pos=.5,right]{$e^4_1$};
\draw[ray] (3,1)--(2,1) node[pos=0,right]{$z_1$} node[pos=1,left]{$z_2$};
\draw[ray] (2,1)--(2,0) node[pos=.5,left]{$e^2_1$};
\draw[ray] (2,0)--(3,0) node[pos=0,left]{$z_3$} node[pos=.5,below] {$e^3_1$} node[pos=1,right]{$z_4$};
\draw[ray] (3,0)--(3,1) node[pos=.5,right]{$e^4_1$};
\node at (2.5,1.5){$B^2_1$};
\node at (2.5,.5){$B^1_1$};

\draw[ray] (4,1)--(4,2) node[pos=.5,left]{$e^2_0$} node[pos=1,left]{$z_3$};
\draw[ray] (4,2)--(5,2) node[pos=.5,above]{$e^3_1$} node[pos=1,right]{$z_4$};
\draw[ray] (5,2)--(5,1) node[pos=.5,right]{$e^4_2$};
\draw[ray] (5,1)--(4,1) node[pos=0,right]{$z_1$} node[pos=1,left]{$z_2$};
\draw[ray] (4,1)--(4,0) node[pos=.5,left]{$e^2_2$};
\draw[ray] (4,0)--(5,0) node[pos=0,left]{$z_3$} node[pos=.5,below] {$e^3_2$} node[pos=1,right]{$z_4$};
\draw[ray] (5,0)--(5,1) node[pos=.5,right]{$e^4_2$};
\node at (4.5,1.5){$B^2_2$};
\node at (4.5,.5){$B^1_2$};

\end{tikzpicture}
\caption{\label{p4}}
\end{figure} 

Our notation is closely related to that used in \cite{wright2012}. In \cite{wright2012}, given a positive integer $N$ and a $n$-by-$4$ matrix $A$, one can define an abelian branched cover of $P$, denoted as $M_N(A)$. In our notation, this half-translation surface becomes $M({span}(\mathbf{q}_1,\mathbf{q}_2,\mathbf{q}_3,\mathbf{q}_4),\\(\mathbf{q}_1,\mathbf{q}_2,\mathbf{q}_3,\mathbf{q}_4))$, where $\mathbf{q}_j\in(\mathbb{Z}/N)^n$ are the column vectors of $A$.\\

Now we describe the action of the deck group $G$ on $M(G,\mathbf{g})$. An element $h\in G$ sends $B^k_g$ to $B^k_{hg}$ and $e^k_g$, to $e^k_{hg}$. The deck group action induces a right $G$-action on $H^1(M,\Sigma;\mathbb{C})$ that makes it a right $G$-module.\\

\section{Affine diffeomorphisms}

From now on we assume that $G$ is abelian, though many of our arguments work for any finite group. At the end of this section we will point out the modifications required in the non-Abelian case.\\

In this section we will calculate $\mathbf{Aff}=\mathbf{Aff}(M(G,\mathbf{g}))$, as well as the Veech group.

 Our method is inspired by the coset graph description used in \cite{sch04}. One distinction between the two approaches is that we consider the whole affine diffeomorphism group while \cite{sch04} considers the Veech group. Fixing $G$, let $V$ be the set of all 4-tuples of elements in $G$: $\mathbf{h}=(h_1,h_2,h_3,h_4)$ such that $\{h_1,h_2,h_3,h_4\}$ generates $G$ and $h_1h_2h_3h_4=1$. Each 4-tuple in $V$ is associated with a square-tiled surface $M(G,\mathbf{h})$ which is equipped with a cell decomposition labeled as in figure~\ref{p2}. An element $F$ in $\mathbf{Aff}$ induces an automorphism of the deck group $G$ by $g\mapsto FgF^{-1}$, i.e. there is a group homomorphism $\mathbf{Aff}\rightarrow {Aut}(G)$. We denote the kernel of this homomorphism as $\Gamma$. Because ${Aut}(G)$ is a finite group, $\Gamma$ is a subgroup of $\mathbf{Aff}$ with finite index. \\

We will show that all orientation preserving affine diffeomorphisms between various surfaces $M(G,\mathbf{h})$ with $G$ fixed and $\mathbf{h}$ varying in $V$, that preserves $\Sigma$ are compositions of a finite set of affine diffeomorphisms. We call this set the set of basic affine diffeomorphisms, and we will describe them below. In our discussion we will be dealing with both translation surfaces and half translation surface surfaces. It will be convenient to view the derivative of an affine diffeomorphism as an element of $PGL(2,\mathbb{R})=GL(2,\mathbb{R})/\{\pm I\}$. We will call an affine translation diffeomorphism a half translation equivalence when its derivative is $1$ in $PGL(2,\mathbb{G})$. \\

Now we define four of the five classes of the basic affine diffeomorphisms:
\begin{enumerate}[(i)]
\item Rotation: For any $(h_1,h_2,h_3,h_4)\in V$, let $t_{(h_1,h_2,h_3,h_4)}$ be the map from  $M(G,(h_2,h_3,h_4,h_1))$ to $M(G,(h_1,h_2,h_3,h_4))$ that sends $B^1_e$ of $M(G,(h_2,h_3,h_4,h_1))$ to $B^1_e$ of $M(G,(h_1,h_2,h_3,h_4))$ by rotating counterclockwise by $\pi/2$. 
\item Deck transformation: For any $(h_1,h_2,h_3,h_4)\in V$, $g\in G$, let $r_{g,\mathbf{h}}$ be the deck transformation $g$ in $M(G,\mathbf{h})$. Its derivative is the identity.
\item Interchange of $B^1$ and $B^2$: For any $(h_1,h_2,h_3,h_4)\in V$, let $f_{(h_1,h_2,h_3,h_4)}$ be the map from $M(G, (h_2,h_1,h_1^{-1}h_4h_1, h_2h_3h_2^{-1}))$ to $M(G,(h_1,h_2,h_3,h_4))$ which interchanges $B^1_g$ and $B^2_g$ by a rotation of $\pi$.
\item Relabeling: For any $(h_1,h_2,h_3,h_4)\in V$, $\psi\in{Aut}(G)$, let $m_\psi$ be the map from $M(G,\mathbf{h})$ to $M(G,\psi(\mathbf{h}))$ that sends $B^j_g$ to $B^j_{\psi(g)}$. Its derivative is the identity.
\end{enumerate}

We claim that:
\begin{lem}
Any half translation equivalence from $M(G,\mathbf{h})$ to $M(G,\mathbf{h'})$ can be written as composition of basic affine diffeomorphisms $t^2$, $r$, $f$ and $m$.
\end{lem}
\begin{proof}
Let $F_0$ be such a half translation equivalence. Denote the unit element of $G$ as $e$. By our assumption, $F_0$ preserves $\Sigma$, hence it is a permutation of unit squares that tiled $M$ and $M'$. More precisely, $F_0$ is completely determined by the following data: i) the induced automorphism $\psi$ of deck group, ii) a number $j=1$ or $2$, which is $1$ when $B^1$ and $B^2$ are taken to themselves and $2$ when they are interchanged, an element $g\in G$, such that $F_0(B^1_e)=B^j_g$, and iii) a number $s$, which is $0$ if $F_0^{-1}(e^1_g)$ is $e^1_e$, $1$ if $F_0^{-1}(e^1_g)$ is $e^3_e$. Hence, $F_0=m_{\psi}t^{2s}f^{j-1}r_g$.\end{proof}

For general orientation preserving affine diffeomorphism $F$, its derivative $DF$ will be in $PSL(2,\mathbb{Z})$. We add another class of basic affine diffeomorphisms:\\
\begin{enumerate}[(v)]
\item Shearing: For any $(h_1,h_2,h_3,h_4)\in V$, let $s_{(h_1,h_2,h_3,h_4)}$ be a map from $M(G,(h_1h_2h_1^{-1},h_1,h_3,h_4))$ to
$M(G,(h_1,h_2,h_3,h_4))$ that sends $e^3_1$ to $e^3_1$ and has derivative $\left( \begin{array}{cc}
1 & -1\\
0 &  1 \end{array} \right)$.
\end{enumerate}

Because the derivative of $s$ and $t$ generate $PSL(2,\mathbb{Z})$, by successively composing with $s$ and $t$ we can reduce to the case when the derivative is the identity. Hence given any $\mathbf{h}, \mathbf{h'}\in V$, any affine diffeomorphism from $M(G,\mathbf{h})$ to $M(G,\mathbf{h'})$ that sends $\Sigma$ to $\Sigma$, or more specifically, any element in $\mathbf{Aff}$, is a composition of the five classes of maps described above. Because $m$ commutes with other 4 classes of diffeomorphisms, i.e. 
$$m_\psi t_{\mathbf{h}}=t_{\psi(\mathbf{h})} m_\psi$$
$$m_\psi r_{g,\mathbf{h}}=r_{\psi(g),\psi(\mathbf{h})} m_\psi$$
$$m_\psi f_{\mathbf{h}}=f_{\psi(\mathbf{h})} m_\psi$$ 
$$m_\psi s_{\mathbf{h}}=s_{\psi(\mathbf{h})} m_\psi$$
any $F\in\mathbf{Aff}$ can be written as $F=F_1m_\psi$ where $F_1$ is a composition of $t$, $s$, $r$ and $f$, while $\psi$ is the automorphism of deck group induced by $F$. Hence, elements in $\Gamma$ can be written as successive compositions of $t$, $r$, $f$ and $s$.\\

As in \cite{sch04}, we will define a directed graph $D$ with vertex set $V$. Each element $\mathbf{h}\in V$ corresponding to a surface $M(G,\mathbf{h})$. An edge in the graph corresponds to a basic affine diffeomorphism between two $M(G,\mathbf{h})$. Paths starting starting and ending at $M(G,\mathbf{g})$ correspond to elements in $\mathbf{Aff}$. Now the fact that any affine diffeomorphism is a successive composition of $t$, $s$, $r$, $f$, and $m$ means that the map from the set of such paths to $\mathbf{Aff}$ is surjective. Similarly, let $D_0$ be graph $D$ with those edges corresponding to $m$ removed, then the set of paths starting and ending at $\mathbf{g}$ in $D_0$ maps surjectively to $\Gamma$.\\ 

Consider the example for which $G=\mathbb{Z}/6$ and $\mathbf{g}=(1,1,1,3)$. This example is the Ornithorynque (c.f.\cite{fm08}).
 In the following figure we give the connected component of $D$ that contains $\mathbf{g}$, with loops corresponding to deck transformation (i.e. all the $r$ arrows) omitted: 
 
\begin{figure}[H]
\begin{tikzpicture}[scale=2.7]
\node at (1,2){$(1,1,1,3)$};
\node at (2,2){$(1,1,3,1)$};
\node at (3,2){$(1,3,1,1)$};
\node at (4,2){$(3,1,1,1)$};
\node at (1,1){$(5,5,5,3)$};
\node at (2,1){$(5,5,3,5)$};
\node at (3,1){$(5,3,5,5)$};
\node at (4,1){$(3,5,5,5)$};
\draw[->](1,1.9)--(1,1.1) node[pos=.5,left]{$m$};
\draw[->](2,1.9)--(2,1.1) node[pos=.5,left]{$m$};
\draw[->](3,1.9)--(3,1.1) node[pos=.5,left]{$m$};
\draw[->](4,1.9)--(4,1.1) node[pos=.5,left]{$m$};
\draw[->](1.1,1.1)--(1.1,1.9) node[pos=.5,right]{$m$};
\draw[->](2.1,1.1)--(2.1,1.9) node[pos=.5,right]{$m$};
\draw[->](3.1,1.1)--(3.1,1.9) node[pos=.5,right]{$m$};
\draw[->](4.1,1.1)--(4.1,1.9) node[pos=.5,right]{$m$};

\draw[->](1.7,1)--(1.3,1)node[pos=.5,above]{$t$};
\draw[->](2.7,1)--(2.3,1)node[pos=.5,above]{$t$};
\draw[->](3.7,1)--(3.3,1)node[pos=.5,above]{$t$};
\draw[-](.7,1)--(.5,1);
\draw[-](.5,1)--(.5,0);
\draw[-](.5,0)--(4.5,0)node[pos=.5,below]{$t$};
\draw[-](4.5,0)--(4.5,1);
\draw[->](4.5,1)--(4.3,1);
\draw[->](1.7,2)--(1.3,2)node[pos=.5,above]{$t$};
\draw[->](2.7,2)--(2.3,2)node[pos=.5,above]{$t$};
\draw[->](3.7,2)--(3.3,2)node[pos=.5,above]{$t$};
\draw[-](.7,2)--(.5,2);
\draw[-](.5,2)--(.5,3);
\draw[-](.5,3)--(4.5,3)node[pos=.5,above]{$t$};
\draw[-](4.5,3)--(4.5,2);
\draw[->](4.5,2)--(4.3,2);

\draw[-](.8,.9)--(.8,.6);
\draw[-](.8,.6)--(1.05,.6)node[pos=.5,below]{$s$};
\draw[->](1.05,.6)--(1.05,.9);
\draw[-](1.95,.9)--(1.95,.6);
\draw[-](1.95,.6)--(2.2,.6)node[pos=.5,below]{$s$};
\draw[->](2.2,.6)--(2.2,.9);

\draw[-](3.1,.9)--(3.1,.6);
\draw[-](3.1,.6)--(3.9,.6)node[pos=.5,above]{$s$};
\draw[->](3.9,.6)--(3.9,.9);
\draw[-](3.95,.9)--(3.95,.55);
\draw[-](3.95,.55)--(3.05,.55)node[pos=.6,below]{$s$};
\draw[->](3.05,.55)--(3.05,.9);
\draw[-](3.1,2.1)--(3.1,2.4);
\draw[-](3.1,2.4)--(3.9,2.4)node[pos=.5,below]{$s$};
\draw[->](3.9,2.4)--(3.9,2.1);
\draw[-](3.95,2.1)--(3.95,2.45);
\draw[-](3.95,2.45)--(3.05,2.45)node[pos=.6,above]{$s$};
\draw[->](3.05,2.45)--(3.05,2.1);

\draw[-](.8,2.1)--(.8,2.4);
\draw[-](.8,2.4)--(1.05,2.4)node[pos=.5,above]{$s$};
\draw[->](1.05,2.4)--(1.05,2.1);
\draw[-](1.95,2.1)--(1.95,2.4);
\draw[-](1.95,2.4)--(2.2,2.4)node[pos=.5,above]{$s$};
\draw[->](2.2,2.4)--(2.2,2.1);

\draw[-](2.9,.9)--(2.9,.25);
\draw[-](2.9,.25)--(4.1,.25)node[pos=.5,below]{$f$};
\draw[->](4.1,.25)--(4.1,.9);
\draw[-](4.05,.9)--(4.05,.3);
\draw[-](4.05,.3)--(2.95,.3)node[pos=.4,above]{$f$};
\draw[->](2.95,.3)--(2.95,.9);
\draw[-](2.9,2.1)--(2.9,2.75);
\draw[-](2.9,2.75)--(4.1,2.75)node[pos=.5,above]{$f$};
\draw[->](4.1,2.75)--(4.1,2.1);
\draw[-](4.05,2.1)--(4.05,2.7);
\draw[-](4.05,2.7)--(2.95,2.7)node[pos=.4,below]{$f$};
\draw[->](2.95,2.7)--(2.95,2.1);

\draw[-](1.1,.9)--(1.1,.5);
\draw[-](1.1,.5)--(1.9,.5)node[pos=.5,below]{$f$};
\draw[->](1.9,.5)--(1.9,.9);
\draw[-](1.85,.9)--(1.85,.55);
\draw[-](1.85,.55)--(1.15,.55)node[pos=.5,above]{$f$};
\draw[->](1.15,.55)--(1.15,.9);
\draw[-](1.1,2.1)--(1.1,2.5);
\draw[-](1.1,2.5)--(1.9,2.5)node[pos=.5,above]{$f$};
\draw[->](1.9,2.5)--(1.9,2.1);
\draw[-](1.85,2.1)--(1.85,2.45);
\draw[-](1.85,2.45)--(1.15,2.45)node[pos=.5,below]{$f$};
\draw[->](1.15,2.45)--(1.15,2.1);

\end{tikzpicture}
\caption{\label{pn}}
\end{figure} 

$\mathbf{Aff}$ of the Ornithorynque can be calculated by finding loops in $D$ that start and end at $\mathbf{g}=(1,1,1,3)$. In this case, $\mathbf{Aff}$ is generated by $tf$, $s$, and $r_{2,(1,1,1,3)}$.\\

The Veech group is the subgroup of $SL(2,\mathbb{R})$ generated by the derivative of these generators.\\

We will show in the next section that the subspaces $H^1(\rho)$ in Theorem 1.1 are invariant under $\Gamma$.\\

This method of calculating $\mathbf{Aff}$ does not use the fact that $G$ is abelian in any way. When $G$ is non-abelian, we can define $\Gamma'$ as the set of elements in $\mathbf{Aff}(M)$ that induce an inner automorphism on $G$, i.e. the set of elements in $\mathbf{Aff}(M)$ that can be written as compositions of $t$, $s$, $r$, $f$ as well as $m_\psi$ where $\psi$ is an element of an inner automorphism of $G$. The argument in the next section will show that $H^1(\rho)$ are invariant under $\Gamma'$.

\section{Invariant decomposition of relative cohomology}

We will now prove Theorem 1.1. Furthermore, we will show that $H^1(\rho)$ are invariant under $\Gamma$ described in the previous section. The proof below can also be viewed as calculation of cohomology with twisted coefficients of $P$, c.f. \cite{hatcher2002algebraic}\\

\begin{proof}[Proof of Theorem 1.1]

Let $M=M(G,\mathbf{g})$, $\Sigma\subset M$ as in Section 1, consider the relative cellular cochain complex 
$$0\rightarrow C^1(M,\Sigma;\mathbb{C})\rightarrow C^2(M,\Sigma;\mathbb{C})$$
 We can identify $C^1(M,\Sigma;\mathbb{C})$ with $(\mathbb{C}[G])^4$ as a right-$G$ module by writing $m\in C^1(M,\Sigma;\mathbb{C})$ as 
 $$(\sum_g m(e^1_g)g^{-1}, \sum_g m(e^2_g)g^{-1},\sum_g m(e^3_g)g^{-1},\sum_g m(e^4_g)g^{-1})\in(\mathbb{C}[G])^4$$
We can also identify $C^2(M,\Sigma;\mathbb{C})$ with $(\mathbb{C}[G])^2$ as a right-$G$ module by writing $n\in C^2(M,\Sigma;\mathbb{C})$ as 
 $$(\sum_g n(B^1_g)g^{-1}, \sum_g n(B^2_g)g^{-1})\in(\mathbb{C}[G])^2$$ 
 Then, the coboundary map from $C^1$ to $C^2$ becomes 
\begin{equation}
d^1(a,b,c,d)=(a+b+c+d,a+g_2b+g_2g_3c+g_2g_3g_4d)
\end{equation}
Hence:
\begin{equation}
H^1(M,\Sigma;\mathbb{C})=\{(a,b,c,d)\in (\mathbb{C}[G])^4:a+b+c+d=a+g_2b+g_2g_3c+g_2g_3g_4d=0\}
\end{equation}
Because $\mathbb{C}[G]$ is semisimple \cite{ser}, it splits into simple algebras $\mathbb{C}[G]=\bigoplus_{\rho\in\Delta} D_\rho$, where $D_\rho$ is the simple subalgebras of $\mathbb{C}[G]$ corresponding to irreducible representation $\rho$. The splitting of the algebra gives a splitting of the cochain complex of $\mathbb{C}[G]$-modules $0\rightarrow C^1\rightarrow C^2$, hence a splitting of the cohomology:
\begin{equation}
H^1(M,\Sigma;\mathbb{C})=\bigoplus_{\rho\in\Delta} H^1(\rho),\;  H^1(\rho)=\{(a,b,c,d)\in D_\rho^4:d^1(a,b,c,d)=0\}
\end{equation}

Let $L$ be the right ideal generated by $\{g_2-1,g_2g_3-1,g_2g_3g_4-1\}$.
The image of $d^1$ in $C^2(M,\Sigma;\mathbb{C})$ is 
$$(1,1)\mathbb{C}[G]\oplus(0,1)L\subset (\mathbb{C}[G])^2=C^2(M,\Sigma;\mathbb{C})$$
This is because
\begin{align*} 
d^1(a,b,c,d)&=(a+b+c+d,a+g_2b+g_2g_3c+g_2g_3g_4d)\\
&=(a+b+c+d,a+b+c+d)+(0,(g_2-1)b)+(0,(g_2g_3-1)c)+(0,(g_2g_3g_4-1)d)\\
&=(a+b+c+d)(1,1)+((g_2-1)b+(g_2g_3-1)c+(g_2g_3g_4-1)d)(0,1)
\end{align*}

Now we show that $\mathbb{C}[G]=\mathbb{C}\oplus L$. Because $(1-a)+(1-b)a=1-ba$, if $g$ is a product of elements in $\{g_2,g_2g_3,g_2g_3g_4\}$ then $1-g\in L$. Also, because $M$ is connected, $\{g_2,g_2g_3,g_2g_3g_4\}$ generates $G$, hence $L$ is generated by all elements of the form $1-g$ for any $g\in G$, therefore $\mathbb{C}[G]=\mathbb{C}\oplus L$, where $\mathbb{C}$ is the trivial sub-algebra generated by $\sum_{g\in G}g$ \cite{ser}. Because $\mathbb{C}[G]$ is semisimple, 
$$H^1(M,\Sigma;\mathbb{C})=\ker(d^1)\rightarrow C^1(M,\Sigma;\mathbb{C})\rightarrow im(d^1)$$
splits, hence we have 
\begin{equation}
H^1(M,\Sigma;\mathbb{C})=(\mathbb{C}[G])^4/(\mathbb{C}[G]\oplus L)=(\mathbb{C}[G])^2\oplus\mathbb{C}
\end{equation}
Therefore, as $G$-module $H^1(\rho)\cong\mathbb{C}^3$ when $\rho$ is the trivial representation, $H^1(\rho)\cong D_\rho^2$ if otherwise. Because $G$ is abelian, $\dim_\mathbb{C} D_\rho=1$, so $\dim_\mathbb{C}H^1(\rho)=3$, if $\rho$ is trivial, $\dim_\mathbb{C}H^1(\rho)=2$, otherwise.\\

In the previous section we described elements in $\mathbf{Aff}$ as compositions of elementary affine diffeomorphisms $t_\mathbf{h}$, $s_\mathbf{h}$, $r_{g,\mathbf{h}}$, $f_{\mathbf{h}}$ and $m_\psi$, and elements in $\Gamma$ as compositions of elementary affine diffeomorphisms $t_\mathbf{h}$, $s_\mathbf{h}$, $r_{g,\mathbf{h}}$ and $f_{\mathbf{h}}$. We will show the invariance of $H^1(\rho)$ under $\Gamma$ by explicitly describing the action of elementary affine diffeomorphisms. The maps from $H^1(M(G,\mathbf{h}),\Sigma;\mathbb{C})$ to some $H^1(M(G,\mathbf{h}),\Sigma;\mathbb{C})$ induced by $t_\mathbf{h}$, $s_\mathbf{h}$, $r_{g,\mathbf{h}}$, $f_{\mathbf{h}}$ 
 are as follows: 
\begin{equation}
t_\mathbf{h}^*([a,b,c,d])=[b,c,d,a]
\end{equation}
\begin{equation}
s_\mathbf{h}^*([a,b,c,d])=[-h_1a,a+b,c,d+h_1a]
\end{equation}
\begin{equation}
r_{g,\mathbf{h}}^*([a,b,c,d])=[ag,bg,cg,dg]
\end{equation}
\begin{equation}
f_\mathbf{h}^*([a,b,c,d])=[-a,-h_2h_3h_4d,-h_2h_3c,-h_2b]
\end{equation}

From equation (3) we know that they all preserve decomposition $H^1(*,\Sigma;\mathbb{C})=\bigoplus_\rho H^1(\rho)$, hence all summands $H^1(\rho)$ are invariant under $\Gamma$.\\

Furthermore, $m_\psi$ is a diffeomorphism from $M(G,\psi^{-1}\mathbf{h})$ to $M(G,\mathbf{h})$, and the map it induced from $H^1(M(G,\mathbf{h}),\Sigma;\mathbb{C})$ to $H^1(M(G,\psi(\mathbf{h})),\Sigma;\mathbb{C})$ is
\begin{align*}
m_\psi^*([\sum_{g\in G} &a_gg,\sum_{g\in G} b_gg,\sum_{g\in G} c_gg,\sum_{g\in G}d_gg])\\
&=[\sum_{g\in G} a_g\psi^{-1}(g),\sum_{g\in G} b_g\psi^{-1}(g),\sum_{g\in G} c_g\psi^{-1}(g),\sum_{g\in G}d_g\psi^{-1}(g)]
\end{align*}
 which, according to equation (3), would send $H^1(\rho)$ to $H^1(\psi^{-1}\rho)$. In other words, elements in $\mathbf{Aff}$ permute $H^1(\rho)$.\\
\end{proof}

\begin{rem} In certain situations $\Gamma=\mathbf{Aff}$. This happens when the $g_j$ are all of different order, or when $G$ is $\mathbb{Z}/n$, $n\geq 4$ and $\mathbf{g}=(1,1,1,n-3)$. In these cases $H^1(\rho)$ are all invariant under $\mathbf{Aff}$. Our argument here is similar to, but not completely the same as those used in \cite{my09}. \end{rem}

\section{The signature of the Hodge form}

Now we define and calculate the signature of an invariant Hermitian form on $H^1(\rho)$ as in \cite{thruston98} and \cite{dm86}. \\

The Hodge form $A_G$ on $H^1(M,\Sigma;\mathbb{C})$ is defined as $1\over 2i$ of the cup product with coefficient pairing $\mathbb{C}\otimes\mathbb{C}\rightarrow \mathbb{C}:z\otimes z'\mapsto z\overline{z'}$ on $(M,\Sigma)$ 
$$H^1(M,\Sigma;\mathbb{C})\times H^1(M,\Sigma;\mathbb{C})\xrightarrow{\smile} H^2(M,\Sigma;\mathbb{C}\otimes\mathbb{C})\rightarrow H^2(M,\Sigma;\mathbb{C})=\mathbb{C}$$
In other words, 
\begin{equation}
\begin{split}
A_G([a,b,c,d],[a',b',c',d'])={1\over 4i}((b,a')_G-(a,b')_G+(d,c')_G-(c, d')_G\\
-(h_2b,a')_G+(a,h_2b')_G-(h_2h_3h_4d,h_2h_3c')_G+(h_2h_3c,h_2h_3h_4d')_G)
\end{split}
\end{equation}
\noindent where  $(\cdot,\cdot)_G$ is the positive definite Hermitian norm on $\mathbb{C}[G]$ defined as 
$$(\sum_ga_gg^{-1},\sum_gb_gg^{-1})_G=\sum_ga_g\overline{b_g}$$
Alternatively, if elements in $H^1(M,\Sigma;\mathbb{C})$ are represented by closed differential forms, $A_G$ can be written as $A_G(\alpha,\beta)={1\over 2i}\int_M\alpha\wedge\overline{\beta}$.\\

$A_G$ is called the area form in \cite{thruston98}, because when $\omega$ defines a flat structure on $M$, the signed area of this flat structure is $A_G(\omega,\overline{\omega})$. Unlike the Hodge norm defined by Eskin-Mirzakhani-Mohammadi, $A_G$ vanishes on the $rel$ subspace.\\

By definition $A_G$ is invariant under the $\Gamma$-action. Furthermore, from (9) and the fact that different $D_\rho$ are orthogonal under $(\cdot,\cdot)_G$, we know that $H^1(\rho)$ for different representation $\rho$ are orthogonal to each other under $A_G$. When $\rho$ is the trivial representation, $A_G=0$ on $H^1(\rho)$.\\

Now we assume $\rho$ to be a non-trivial representation. Because we will deal with Hermitian forms that may be degenerate, we denote the signature of a Hermitian form as $(n_0,n_+,n_-)$, where $n_0$, $n_+$, $n_-$ are the number of 0, positive and negative eigenvalues respectively.\\

We will prove the following theorem:

\begin{thm} When $\rho\neq 1$, the signature of the area form $A_G$ on $H^1(\rho)$, is $(n_0,{\theta_2\over 2\pi}-1,{\theta_1\over 2\pi}-1)=(4-{\theta_1+\theta_2\over 2\pi},{\theta_2\over 2\pi}-1,{\theta_1\over 2\pi}-1)$, where $\theta_1=\sum_{j=1}^4 arg(\rho(g_j))$, $\theta_2=\sum_{j=1}^4 arg(\rho(g_j)^{-1})$. The number $n_0$ is also the number of indices $j$ such that $\rho(g_j)=1$.
\end{thm}

\begin{proof}
In the case when $\rho(g_1g_2)=\rho(g_2g_3)=1$ then $a=c$, $b=d$, and $A_G$ is $2|G|$ times the area of parallelogram with side $a$ and $b$, i.e. proportional to the cross product of two vectors on the complex plane, which has signature $(0,1,1)=(4-{\theta_1+\theta_2\over 2\pi},{\theta_2\over 2\pi}-1,{\theta_1\over 2\pi}-1)$.\\

Now we consider the case when $\rho(g_1g_2)\neq 1$ or $\rho(g_2g_3)\neq 1$. Without losing generality we assume $\rho(g_1g_2)\neq 1$. From equations (1), (3) we know that any $(a,b,c,d)\in H^1(\rho)$ satisfies
\begin{equation}
a+b+c+d=0
\end{equation}
\begin{equation}
a+\rho(g_2)b+\rho(g_2g_3)c+\rho(g_2g_3g_4)d=0
\end{equation}
Because $\rho(g_3g_4)=\rho(g_1g_2)^{-1}\neq 1$, we can solve $(b,d)$ from these two equations as linear functions of $(a,c)$, i.e. rewrite these equations can be written as $(b,d)=(a,c)A$ where $A$ is a 2-by-2 matrix.\\

Consider subspaces $H^1_a=\{(a,b,0,d)\in H^1(\rho)\}$, $H^1_{a'}=\{(0,b,c,d)\in H^1(\rho)\}$, then by the previous arguments $\dim(H^1_a)=\dim(H^1_{a'}=1$ and $H^1(\rho)=H^1_a\oplus H^1_{a'}$. We will show that they are orthogonal subspaces under $A_G$. For any $(a,b,0,d)\in H^1_a(\rho), (0,b',c',d')\in H^1_{a'}(\rho)$, because $(*,*)_G$ is $G$-invariant and $d^1(a,b,0,d)=d^1(0,b',c',d')=0$, we have
\begin{align*}
A_G([a,b,0,d],[0,b',c',d']) &= {1\over 2i}(-(a,b')_G+(d,c')_G+(a,h_2b')_G-(h_2h_3h_4d,h_2h_3c')_G)\\
&= {1\over 2i}((b,b')_G+(d,b')_G-(d,b')_G-(d,d')_G-(h_2b,h_2b')_G\\
-&(h_2h_3h_4d,h_2b')_G+(h_2h_3h_4d,h_2h_3h_4d')_G+(h_2h_3h_4d,h_2b')_G)\\
&= 0
\end{align*}
\noindent In other words, $A_G$ is diagonalized under $H^1(\rho)=H^1_a\oplus H^1_{a'}$.\\

Now we show that the signature of $A_G$ on $H^1_a$ is $(3-{\theta_{1a}+\theta_{2a}\over 2\pi},{\theta_{2a}\over 2\pi}-1,{\theta_{1a}\over 2\pi}-1)$, where $\theta_{1a}= arg(\rho(g_1))+arg(\rho(g_2))+arg(\rho(g_3g_4))$, $\theta_{2a}= arg(\rho(g_1)^{-1})+arg(\rho(g_2)^{-1})+arg(\rho(g_3g_4)^{-1})$. From (10) and (11) we know that 
\begin{equation}
H^1_a=\{t(\rho(g_2)-\rho(g_1^{-1}),\rho(g_1^{-1})-1,0,\rho(g_2)-1):t\in D_\rho\}
\end{equation}
If $\rho(g_2)=0$, equation (12) becomes $H^1_a=\{(t,-t,0,0):t\in D_\rho\}$. If $\rho(g_1)=0$, (12) becomes $H^1_a=\{(t,0,0,-t):t\in D_\rho\}$. In both cases the signature is $(1,0,0)=(3-{\theta_{1a}+\theta_{2a}\over 2\pi},{\theta_{2a}\over 2\pi}-1,{\theta_{1a}\over 2\pi}-1)$. If neither $\rho(g_1)$ nor $\rho(g_2)$ is 1, $\theta_{1a}+\theta_{2a}=6\pi$, and $\theta_{1a}$ is either $2\pi$ or $4\pi$. From (9) and (12) we know that $A_G$ is positive on $H^1_a$ definite when $\theta_{1a}=2\pi$ and negative definite on $H^1_a$ when $\theta_{1a}=4\pi$, i.e. the signature of $A_G$ on $H^1_a$ is $(3-{\theta_{1a}+\theta_{2a}\over 2\pi},{\theta_{2a}\over 2\pi}-1,{\theta_{1a}\over 2\pi}-1)$.\\

We can calculate the signature of $A_G$ on $H^1_{a'}$ similarly. Because $A_G$ is diagonalized under $H^1(\rho)=H^1_a\oplus H^1_{a'}$, the $n_0$, $n_+$ and $n_-$ of $A_G$ on $H^1(\rho)$ can be obtained by adding the $n_0$, $n_+$ and $n_-$ of $A_G$ on $H^1_a$ and $H^1_{a'}$.
\end{proof}

\section{The subgroup $\Gamma_1$ and triangle groups}

 In this section we introduce a subgroup $\Gamma_1$ of $\Gamma$ of finite index, which is easier to work with than $\Gamma$. In Section 6, we will give a criteria for non-discreteness of the action of $\Gamma$ by analyzing the action of this subgroup of finite index.\\
 
 There is a homomorphism $D: \mathbf{Aff}\rightarrow SL(2,\mathbb{Z})$ that sends an affine diffeomorphism to its derivative. Because elements of $\mathbf{Aff}$ preserves $\Sigma$, $\ker(D)$ is finite. Consider two elements in $\Gamma$ which are liftings of the horizontal and vertical Dehn twists of the pillowcase 
$$\gamma_1=s_{(g_1,g_2,g_3,g_4)}s_{(g_2,g_1,g_3,g_4)}$$ 
$$\gamma_2=t_{(g_1,g_2,g_3,g_4)}s_{(g_2,g_3,g_4,g_1)}s_{(g_3,g_2,g_4,g_1)}t^{-1}_{(g_1,g_2,g_3,g_4)}$$
 $D\gamma_1$ and $D\gamma_2$ generates the level 2 congruence subgroup of $SL(2,\mathbb{Z})$, hence the group generated by them is a subgroup of $\mathbf{Aff}$ of finite order. From (5), (6) we have:
\begin{equation}
\gamma_1^*(a,b,c,d)=(g_1g_2a, b+a-g_1a,c,d+g_1a-g_1g_2a)
\end{equation}
\begin{equation}
\gamma_2^*(a,b,c,d)=(a+g_2b-g_2g_3b,g_2g_3b,c+b-g_2b,d)
\end{equation}
Denote the group generated by $\gamma_1$ and $\gamma_2$ as $\Gamma_1$. When restricted to $H^1(\rho)$,
\begin{equation}
\gamma_1^*(a,b,c,d)=(\rho(g_1g_2)a, b+a-\rho(g_1)a,c,d+\rho(g_1)a-\rho(g_1g_2)a)
\end{equation}
\begin{equation}
\gamma_2^*(a,b,c,d)=(a+\rho(g_2)b-\rho(g_2g_3)b,\rho(g_2g_3)b,c+b-\rho(g_2)b,d)
\end{equation}
When $\rho=1$ is the trivial representation, the $\Gamma_1$ action on $H^1(\rho)$ is trivial. When $\rho$ is non-trivial, the $\Gamma$ action on $H^1(\rho)=\mathbb{C}^2$ induces an action on $\mathbb{CP}^1$ under projectivization. The map on $\mathbb{CP}^1$ induced by $\gamma_1$ is parabolic if and only if $\rho(g_1g_2)=1$. Similarly, $\gamma_2$ is parabolic if and only if $\rho(g_2g_3)=1$. If they are not parabolic they are elliptic.\\

\begin{rem}Furthermore, when $\rho(g_2)=1$ and all other $\rho(g_j)\neq 1$, the $\Gamma_1$ action $H^1(\rho)$ is not semisimple. We can see this as follows: by equations (15) and (16), $(1,-1,0,0)$ is the only common eigenvector of $\gamma_1^*$ and $\gamma_2^*$ in $H^1(\rho)$, hence $H^1_a(\rho)=\{(t,-t,0,0)\}$ is the only 1-dimensional subspace of $H^1(\rho)$ invariant under $\Gamma_1$, i.e. in this case the $\Gamma_1$ action on $H^1(\rho)$ is not semisimple. $H^1_a(\rho)$ is also the kernel of the projection $H^1(M,\Sigma;\mathbb{C})\rightarrow H^1(M;\mathbb{C})$ restricted to $H^1(\rho)$.\end{rem}

Hence, we have:
\begin{proof}[Proof of Theorem 1.2] 
(1),(3) follows from Theorem 4.1. (2) follows from Remark 3. 
\end{proof}

The Hodge norm $A_G$ on $H^1(\rho)$ induces a metric, hence a geometric structure on the projectivization $\mathbb{P}(H^1(\rho))=\mathbb{CP}^1$ invariant under the $\Gamma$-action. When $A_G$ is positive definite or negative definite, it induces a spherical structure on $\mathbb{CP}^1$. When $A_G$ has signature $(1,0,1)$, it induces a Euclidean structure on $\mathbb{CP}^1-[0:1]$. When $A_G$ has signature $(1,1,0)$, it induces a Euclidean structure on $\mathbb{CP}^1-[1:0]$. Finally, when the signature of $A_G$ is $(0,1,1)$, it induces a hyperbolic structure on a disc $D$ in $\mathbb{P}H^1(\rho)$, which consists of the image $\{\alpha\in H^1(\rho):A(\alpha,\alpha)>0\}$. The case (3) of Theorem 1.2 corresponds to the elliptic and parabolic cases, while case (2) of Theorem 1.2 corresponds to the Euclidean case. In all these cases, we will show that the action of $\Gamma_1$ is through an index-2 subgroup of the triangle group, which proves Theorem 1.3.

\begin{proof}[Proof of Theorem 1.3] 
In the spherical case, all the elements in $\Gamma_1$ are rotations hence has a pair of fixed points. Let $P_1$ be a fixed point of $\gamma_1$, $P_2$ be a fixed point of $\gamma_2$, $P_3$ be a fixed point of $\gamma_2^{-1}\gamma_1$, let $\Delta P_1P_2P_3$ and $\Delta P_1P_2\gamma_1(P_3)$ be spherical triangles formed by the shortest geodesics between these points, then they are related by reflection along $P_1P_2$, and $\gamma_1$ can be presented as reflection along first $P_1P_3$ then $P_1P_2$, while $\gamma_2$ is reflection along $P_2P_3$ followed by reflection along $P_1P_2$, hence $\Gamma_1$ acts through a subgroup of the triangle reflection group $T$ corresponding to $\Delta P_1P_2P_3$. Because $\Gamma_1$ action preserves orientation, its image under the action is contained in the index-2 subgroup consisting of compositions of even number of reflections along the three sides of the triangle.\\

On the other hand, a reflection along $P_1P_2$ then $P_1P_3$ is the same as the action of $\gamma_1^{-1}$, a reflection along $P_1P_2$ then $P_2P_3$ is the same as the action $\gamma_2^{-1}$, a reflection along $P_1P_3$ then $P_2P_3$ is the same as the action $\gamma_2^{-1}\gamma_1$, and a reflection along $P_2P_3$ then $P_1P_3$ is the same as the action $\gamma_1^{-1}\gamma_2$, the composition of any sequence of even number of reflections along $P_1P_2$, $P_2P_3$ and $P_1P_3$ can be written as the action of a composition of $\gamma_1$, $\gamma_2$ and their inverses, hence $Gamma_1$ acts through the index-2 subgroup of $T$ consisting of compositions of even number of reflections.\\

In the euclidean case the argument is the same. In the hyperbolic case, when $\gamma_1$, $\gamma_2$ and $\gamma_2^{-1}\gamma_1$ are all elliptic elements, we choose fixed points $P_1$, $P_2$ and $P_3$ all in the disc $D$ described above, then the same argument would work.\\

Finally, in the hyperbolic case, in which $\gamma_1$, $\gamma_2$ or $\gamma_2^{-1}\gamma_1$ is parabolic, choose the corresponding $P_i$ to be the unique fixed point on $\partial D$, then the sides ending at $P_i$ become geodesic rays in $D$ that approaches $P_i$ on the ideal boundary. The same argument will show that the $\Gamma_1$-action is through an index-2 subgroup of a triangle reflection group with certain angles being $0$.

\end{proof}

Suppose elliptic element $\gamma\in PGL(2,\mathbb{C})$ has a fixed point $P\in\mathbb{CP}^1$, then by conjugating with an element in $PGL(2,\mathbb{C})$ we can make $P=[1,0]$ and $\gamma=\left( \begin{array}{cc}
e^{\sqrt{-1}\theta_1} &  0\\
0 &  e^{\sqrt{-1}\theta_2} \end{array} \right)$, which in the tangent space $T_{P}\mathbb{CP}^1$ induces a rotation of angle $\theta_2-\theta_1$. Hence, an elliptic element in $\Gamma_1$ induces a rotation and the rotation angle is the argument of the quotient of the 2 eigenvalues, hence can be calculated with (15) and (16). The angles of the triangle at $P_1$, $P_2$ and $P_3$ are half of the rotation angles of $\gamma_1$, $\gamma_2^{-1}$ and $\gamma_2^{-1}\gamma_1$ respectively. We will describe these triangle groups and their applications in the next two sections.

\section{The spherical case and polyhedral groups}

In this section we will describe the spherical case, and in the next section we will describe the remaining cases.\\

When $A_G$ is positive definite or negative definite, the generators of $\Gamma_1$, $\gamma_1$ and $\gamma_2$, act as finite order rotations with different fixed points, and their orders are the orders of $\rho(g_1g_2)$ and $\rho(g_2g_3)$ in $\mathbb{C}^*$ respectively, hence by the ADE classification \cite{led} of finite subgroups of $SO(3)$ we know that if both the orders of $\rho(g_1g_2)$ and $\rho(g_2g_3)$ are greater than 5 the action of $\Gamma$ on $H^1(\rho)$ can not factor through a discrete group.\\

\begin{proof}[Proof of Proposition 1.4 and Remark 1]: Let $G$ be the subgroup of $(\mathbb{Z}/120)^3$ spanned by $g_1=(20,0,0)$, $g_2=(0,15,0)$, $g_3=(0,0,12)$, $g_4=(100,105,108)$, $\mathbf{g}=(g_1,g_2,g_3,g_4)$, $M=M(G,\mathbf{g})$, $\rho(g_1)=e^{\pi i/3}$, $\rho(g_2)=e^{\pi i/4}$, $\rho(g_3)=e^{\pi i/5}$, $\rho(g_4)=e^{73i\pi/60}$. Then by Theorem 1.1 and Remark 2 $H^1(\rho)$ is invariant under $\mathbf{Aff}$ with an invariant complement, by Section 4 the Hodge form is positive definite on $H^1(\rho)$, and by the argument above the $\mathbf{Aff}$ action on $H^1(\rho)$ is not discrete. In other words, $M$ and the decomposition $H^1(M,\Sigma;\mathbb{C})=H^1(\rho)\oplus H^1(\overline{\rho})\oplus (\bigoplus_{\rho'\neq \rho,\overline{\rho}}H^1(\rho'))$ satisfy all the conditions mentioned in Proposition 1.4. Furthermore, the tangent space of the $GL(2,\mathbb{R})$ orbit is $H^1(\rho_t)$ where $\rho_t(g_j)=-1$ for $j=1,2,3,4$. Because $\rho_t\neq \rho$, $\rho_t\neq\overline{\rho}$, the decomposition also satisfy conditions in Remark 1.
\end{proof}

Furthermore, by Theorem 1.3 and \cite{cox}, we can list all possible 4-tuples \\$(\rho(g_1),\rho(g_2),\rho(g_3),\rho(g_4))$ such that the $\Gamma_1$ action on $\mathbb{P}(H^1(\rho))=\mathbb{CP}^1$ factors through a finite group. Denote the arguments of $\rho(g_j)$ as $2t_j\pi$, $j=1,\dots 4$, then the $\Gamma_1$ action is finite if and only if $t_1$, $t_2$, $t_3$, $t_4$ is a permutation of one of the 4-tuples in the table below:
\begin{table}[H]
\begin{tabular}{cccc|c}
\hline
 $t_1$&$t_2$&$t_3$&$t_4$& Group\\ \hline
 d/2n&d/2n&(n-d)/2n&(n-d)/2n&Dihedral group\\
 1/12 &1/4 &1/4&5/12&Tetrahedral group\\
 1/24 &5/24&7/24&11/24&Octahedral group\\
 1/60&11/60&19/60&29/60&Icosahedral group\\
 1/6&1/6&1/6&1/2&Tetrahedral group\\
 1/12&1/6&1/6&7/12&Octahedral group\\
 1/30&19/30&1/6&1/6&Icosahedral group\\
1/30&3/10&3/10&11/30&Icosahedral group\\
1/20&3/20&7/20&9/20&Icosahedral group\\
1/15&2/15&4/15&8/15&Icosahedral group\\
1/10&3/10&3/10&3/10&Icosahedral group\\
1/10&1/10&7/30&17/30&Icosahedral group\\
1/10&1/10&1/10&7/10&Icosahedral group\\
7/60&13/60&17/60&23/60&Icosahedral group\\
1/6&1/6&7/30&13/30&Icosahedral group\\
\hline
\end{tabular}
\end{table}

\begin{table}[H]
\begin{tabular}{cccc|c}
\hline
 $t_1$&$t_2$&$t_3$&$t_4$& Group\\ \hline
 (2n-d)/2n&(2n-d)/2n&(n+d)/2n&(n+d)/2n&Dihedral group\\
 7/12 &3/4 &3/4&11/12&Tetrahedral group\\
 13/24 &19/24&17/24&23/24&Octahedral group\\
 31/60&41/60&49/60&59/60&Icosahedral group\\
 1/2&5/6&5/6&5/6&Tetrahedral group\\
 5/12&5/6&5/6&11/12&Octahedral group\\
 5/6&5/6&11/30&29/30&Icosahedral group\\
19/30&7/10&7/10&29/30&Icosahedral group\\
11/20&13/20&17/20&19/20&Icosahedral group\\
7/15&11/15&13/15&14/15&Icosahedral group\\
7/10&7/10&7/10&9/10&Icosahedral group\\
13/30&23/30&9/10&9/10&Icosahedral group\\
3/10&9/10&9/10&9/10&Icosahedral group\\
37/60&43/60&47/60&53/60&Icosahedral group\\
17/30&23/30&5/6&5/6&Icosahedral group\\
\hline
\end{tabular}
\end{table}

Here $d$ and $n$ are positive integers, and the last column shows the discrete subgroup of $SO(3)$ it corresponds to.\\

Examples of $M$ and $\rho$ that satisfies the conditions needed for Proposition 1.4 and Remark 1 can therefore be built from any 4-tuple of positive rational numbers not on the above list that sum up to 1 or 3. For example, $(1/8,1/8,1/8,5/8)$ is not on the list, so let $G=\mathbb{Z}/8$, $g_1=g_2=g_3=1$, $g_4=5$, $\rho(g_1)=\rho(g_2)=\rho(g_3)=e^{\pi i/4}$, $\rho(g_4)=e^{5\pi i/4}$ satisfies those conditions. This is an abelian cover of flat pillowcase that satisfy those conditions with the smallest number of squares. It has $16$ squares in total and is in the stratum $H(3,3,3,3)$.\\

Explicit elements of $N=H^1(\rho)$ and $Re(N)=Re(H^1(\rho))$ can also be calculated by solving (10) and (11). For example, in the abovementioned $G=\mathbb{Z}/8$, $g_1=g_2=g_3=1$, $g_4=5$, $\rho(g_1)=\rho(g_2)=\rho(g_3)=e^{\pi i/4}$, $\rho(g_4)=e^{5\pi i/4}$ case, $\alpha\in H^1(M,\Sigma,\mathbb{R})$, $\alpha(e^1_k)=-2\cos(\pi/8)\cos(k\pi/4)$, $\alpha(e^2_k)=\cos(\pi/8+k\pi/4)$, $\alpha(e^3_k)=0$, $\alpha(e^4_k)=\cos(\pi/8-k\pi/4)$ is an element in $Re(H^1(\rho))$.\\

\section{The hyperbolic and Euclidean cases and triangle groups}

In this section we complete the description of $\Gamma_1$-action by describing the signature $(0,1,1)$, $(1,0,1)$, $(1,1,0)$ cases as triangle groups. In the $(0,1,1)$ case, $\sum_j{arg}(\rho(g_j))=4\pi$. Without losing generality we assume ${arg}(\rho(g_1))+{arg}(\rho(g_2))\geq 2\pi$, ${arg}(\rho(g_2))+{arg}(\rho(g_3))\geq 2\pi$.\\

By calculation based on the observations in Section 5, $\gamma_1$ acts on $D$ as a rotation by ${arg}(\rho(g_1g_2))$ when $\rho(g_1g_2)\neq 1$ and as a parabolic transformation when $\rho(g_1g_2)=1$, $\gamma_2$ acts on $D$ as a rotation by ${arg}(\rho(g_2g_3))$ when $\rho(g_2g_3)\neq 1$ and as a parabolic transform when $\rho(g_2g_3)=1$. Furthermore, $\gamma_1$ and $\gamma_2$ generate an index 2 subgroup of a triangle group, and the angles of the triangle are $|\pi-({arg}(\rho(g_1))+{arg}(\rho(g_2)))/2|$, $|\pi-({arg}(\rho(g_2))+{arg}(\rho(g_3)))/2|$ and $|\pi-({arg}(\rho(g_1))+{arg}(\rho(g_3)))/2|$. \cite{wright2012} has done the same calculation and based on it has calculated the Lyapunov exponents from the area of this triangle.\\

When the signature of $A_G$ is $(1,1,0)$ or $(1,0,1)$, only one $\rho(g_j)$ is equal to $1$. Without losing generality assume $\rho(g_2)=1$, then $(a,b,c,d)\mapsto b/a$ sends $H^1(\rho)$ to $\overline{\mathbb{C}}$, and under this map $\Gamma_1$ acts on $\mathbb{C}=\overline{\mathbb{C}}-\{\infty\}$ as an index-2 subgroup of a Euclidean triangle group. The angles of the triangle are ${arg}(\rho(g_1))/2$, ${arg}(\rho(g_3))/2$ and ${arg}(\rho(g_4))/2$ when 
$${arg}(\rho(g_1))+{arg}(\rho(g_3))+{arg}(\rho(g_4))=2\pi$$
 i.e. when the signature of $A_G$ is $(1,1,0)$. The angles of the triangles are $\pi-{arg}(\rho(g_1))/2$, $\pi-{arg}(\rho(g_3))/2$ and $\pi-{arg}(\rho(g_4))/2$ when 
$${arg}(\rho(g_1))+{arg}(\rho(g_3))+{arg}(\rho(g_4))=4\pi$$
 i.e. when the signature of $A_G$ is $(1,0,1)$.\\

When two of the four $\rho(g_j)$ are equal to $1$,  then the $\Gamma_1$ action on $H^1(\rho)$ factors through a finite abelian group. Finally, when $\rho=1$, the $\Gamma_1$ action on $H^1(\rho)$ is trivial, and the $\Gamma$-action is through a finite permutation group.\\

\bibliographystyle{alpha} 
\bibliography{abelr}

\end{document}